\documentclass[12pt]{article}
\usepackage{amsmath,amssymb,amsthm,comment}

\newtheorem{theorem}{Theorem}[section]
\newtheorem{lemma}[theorem]{Lemma}
\newtheorem{corollary}[theorem]{Corollary}

\newtheorem{example}[theorem]{Example}

\def\barr{\begin{array}}
\def\earr{\end{array}}

\title{On CLT and non-CLT groups}
\author{Marius T\u arn\u auceanu}
\date{March 9, 2024}

\begin{document}

\maketitle

\begin{abstract}
In this note, we prove that for every integer $d\geq 2$ which is not a prime power, there exists a finite solvable group $G$ such that $d\mid |G|$, $\pi(G)=\pi(d)$ and $G$ has no subgroup of order $d$. We also introduce the CLT-degree of a finite group and answer two questions about it.
\end{abstract}

{\small
\noindent
{\bf MSC2020\,:} Primary 20D60; Secondary 20D99, 20E99.

\noindent
{\bf Key words\,:} CLT group, solvable group, Frobenius group.}

\section{Introduction}

All groups considered in this note are finite. A group is said to be CLT if it possesses subgroups of every possible order (that is, it satisfies the Converse of Lagrange's Theorem) and non-CLT
otherwise. It is well-known that CLT groups are solvable (see \cite{11}) and that supersolvable groups are CLT (see \cite{10}). Moreover, the inclusion between the classes of CLT groups
and solvable groups, as well as the inclusion between the classes of supersolvable groups and CLT groups are proper (see, for example, \cite{3}). Recall also the papers \cite{1,2,13} which construct non-CLT groups of order $p^{\alpha}q^{\beta}$ with $p,q$ primes and $\alpha,\beta\in\mathbb{N}^*$.

The first starting point for our work is given by the smallest examples of non-CLT groups, namely $A_4$ and ${\rm SL}(2,3)$. It is well-known that they have no subgroup of order $6$ and $12$, respectively. This leads to the following natural question:
\begin{center}
\textit{Given an integer $d\geq 2$, is it possible to construct a group\\ whose order is divisible by $d$, but having no subgroup of order $d$}?
\end{center}Obviously, the answer to this question is "NO" if $d$ is a prime power. So, in what follows we will assume that $d$ is a composite number and we will denote by $\pi(d)$ the set of primes dividing $d$. Also, for a finite group $G$ we will denote $\pi(G)=\pi(|G|)$.

Our first result shows that the answer to the above question is "YES" for all composite numbers $d$.

\begin{theorem}
For every integer $d\geq 2$ which is not a prime power, there exists a finite solvable group $G$ such that $d\mid |G|$, $\pi(G)=\pi(d)$ and $G$ has no subgroup of order $d$.
\end{theorem}

The second starting point for our work is given by a question of Martino Garonzi on MathOverflow \cite{4}:
\begin{center}
\textit{Is there a constant $c>0$ such that $\frac{D(G)}{\tau(|G|)}>c$ for any finite group $G$}?
\end{center}Here $D(G)$ denotes the number of divisors $d$ of $|G|$ for which there exists a subgroup of $G$ of order $d$ and $\tau(|G|)$ denotes the number of all divisors of $|G|$. Note that this question remained unanswered for more than ten years.

It suggests to consider the function
\begin{equation}
d_{CLT}(G)=\frac{D(G)}{\tau(|G|)}\,,\nonumber
\end{equation}which will be called the \textit{CLT-degree} of $G$. Clearly, $d_{CLT}(G)$ measures the probability of a finite group $G$ to be CLT. It is easy to see that this new function satisfies the following properties:
\begin{itemize}
\item[-] $0<d_{CLT}(G)\leq 1$, for any finite group $G$. Moreover, $d_{CLT}(G)=1$ if and only if $G$ is CLT.
\item[-] $d_{CLT}$ is multiplicative, that is $d_{CLT}(G_1\times G_2)=d_{CLT}(G_1)d_{CLT}(G_2)$, for any finite groups $G_1$, $G_2$ of coprime orders.
\item[-] If $N$ is a normal subgroup of $G$, then 
\begin{equation}
d_{CLT}(G/N)\leq\frac{\tau(|G|)}{\tau(|G|/|N|)}\,d_{CLT}(G).\nonumber
\end{equation}
\item[-] If $|G|=p_1^{n_1}\cdots p_k^{n_k}$ with $k\geq 2$, then 
\begin{equation}
d_{CLT}(G)\geq\frac{\sum_{i=1}^kn_i+2}{\prod_{i=1}^k(n_i+1)}\,.\nonumber 
\end{equation}Moreover, we have equality if and only if $G$ has all proper subgroups of prime power order. Such a group is either cyclic of order $pq$, where $p$ and $q$ are distinct primes, or a semidirect product $C_p^n\rtimes C_q$, where $C_q$ acts irreducibly on $C_p^n$. Thus
\begin{equation}
d_{CLT}(C_p^n\rtimes C_q)=\frac{n+3}{2n+2}\nonumber
\end{equation}and, in particular, $d_{CLT}(A_4)=\frac{5}{6}$\,.
\end{itemize}

Let $\mathcal{G}$ be the class of all finite groups. Our second result is stated as follows. 

\begin{theorem}
The set 
\begin{equation}
{\rm Im}(d_{CLT})=\{d_{CLT}(G)\mid G\in\mathcal{G}\}\nonumber 
\end{equation}is dense in $[0,1]$.
\end{theorem}

Obviously, Theorem 1.2 gives a negative answer to Garonzi's question.

\begin{corollary} 
There is no constant $c>0$ such that $\frac{D(G)}{\tau(|G|)}>c$ for any finite group $G$.
\end{corollary}

We also infer that:

\begin{corollary} 
There is no constant $c<1$ such that $d_{CLT}(G)\leq c$ for any finite non-CLT group $G$.
\end{corollary}Note that an example of a family of finite non-CLT groups $(G_n)_{n\geq 1}$ with $\lim_{n\rightarrow\infty}d_{CLT}(G_n)=1$ will be given at the end of Section 3.
\bigskip

Finally, we formulate a natural open problem related to our results.

\bigskip\noindent{\bf Open problem.} Is it true that $\lim_{n\rightarrow\infty}d_{CLT}(S_n)=0$?
\bigskip

Note that we have
\begin{equation}
d_{CLT}(S_n)\leq\frac{\#iso(S_n)}{\tau(n!)}\leq\frac{\#ccs(S_n)}{\tau(n!)}\,,\nonumber
\end{equation}where $\#iso(S_n)$ and $\#ccs(S_n)$ denote the number of isomorphism/conjugacy classes of subgroups of $S_n$, but we were not able to decide whether these ratios tend to $0$ when $n$ tends to infinity.
\bigskip

Most of our notation is standard and will usually not be repeated here. For basic notions and results on groups we refer the reader to \cite{6,7}.

\section{Proof of Theorem 1.1}

First of all, we present two auxiliary results. Recall that a \textit{Frobenius group} $NH$ with kernel $N$ and complement $H$ can be characterized as a finite group that is a semidirect product of a normal subgroup $N$ by a subgroup $H$ such that $C_N(h)=1$ for every $h\in H\setminus\{1\}$. An important example of such a group is the group ${\rm AGL}(1,q)=\mathbb{F}_q\rtimes\mathbb{F}_q^{\times}$ of affine linear transformations of the finite field $\mathbb{F}_q$.

\begin{lemma}([5])
Let $G$ be a Frobenius group with kernel $N$ and $K$ be a subgroup of $G$. Then one of the following holds:
\begin{itemize}
\item[{\rm 1)}] $K\subseteq N$.
\item[{\rm 2)}] $K\cap N=1$.
\item[{\rm 3)}] $K$ is a Frobenius group with kernel $N\cap K$.
\end{itemize}
\end{lemma}

\begin{lemma}([8])
Let $G$ be a Frobenius group with kernel $N$ and complement $H$. Then $|H|\mid |N|-1$.
\end{lemma}

We are now able to prove our first main result.

\bigskip\noindent{\bf Proof of Theorem 1.1.} 
We will proceed by induction on $d$. For $d=6$ we can choose $G=A_4$. 
Now, let $d\geq 6$ and assume the statement to be true for every $d'<d$. Let $d=p_1^{n_1}\cdots p_k^{n_k}$ be the decomposition of $d$ as a product of distinct prime factors, where $k\geq 2$. We distinguish the following two cases.
\medskip

\noindent\hspace{10mm}{\bf Case 1.} $k=2$
\smallskip

\noindent Then $d=p_1^{n_1}p_2^{n_2}$. Let $a=\exp_{p_2^{n_2}}(p_1)$ and $b=\exp_{p_1^{n_1}}(p_2)$, where $\exp_{p_i^{n_i}}(p_j)$ is the multiplicative order of $p_j$ modulo $p_i^{n_i}$. Then either $a\nmid n_1$ or $b\nmid n_2$. Indeed, if $a\mid n_1$ and $b\mid n_2$, then $p_2^{n_2}\mid p_1^{n_1}-1$ and $p_1^{n_1}\mid p_2^{n_2}-1$, implying that $p_2^{n_2}<p_1^{n_1}$ and $p_1^{n_1}<p_2^{n_2}$, a contradiction. 

Assume that $a\nmid n_1$ and let $r$ be a positive integer such that $(r-1)a<n_1<ra$. Then the Frobenius group 
\begin{equation}
{\rm AGL}(1,p_1^{ra})=C_{p_1}^{ra}\rtimes C_{p_1^{ra}-1}\nonumber 
\end{equation}has no subgroup of order $d$. Indeed, if there is $K\leq{\rm AGL}(1,p_1^{ra})$ with $|K|=d$, then $K$ must be a Frobenius group with kernel of order $p_1^{n_1}$ and complement of order $p_2^{n_2}$ by Lemma 2.1. Then Lemma 2.2 leads to $p_2^{n_2}\mid p_1^{n_1}-1$ and thus $a\mid n_1$, a contradiction. Now, it is clear that ${\rm AGL}(1,p_1^{ra})$ has a subgroup $G$ of order $p_1^{ra}p_2^{n_2}$: this $G$ is easily seen to
be a group that satisfies the desired conclusions with respect to $d=p_1^{n_1}p_2^{n_2}$.\newpage
\medskip

\noindent\hspace{10mm}{\bf Case 2.} $k\geq 3$
\smallskip

\noindent Let $d'=d/p_k^{n_k}$. Then $d'$ is not a prime power and so there exists a finite solvable group $G_1$ such that $d'\mid |G_1|$, $\pi(G_1)=\pi(d')$ and $G_1$ has no subgroup of order $d'$. Let $G=G_1\times C_{p_k^{n_k}}$. It follows that $G$ is solvable, $d\mid |G|$ and $\pi(G)=\pi(d)$. Moreover, $G$ has no subgroup of order $d$. Indeed, if $H\leq G$ has order $d=d'\cdot p_k^{n_k}$, then $H$ possesses a subgroup $H_1$ of order $d'$. Since $G$ is solvable, $H_1$ is contained in a Hall $\pi(d')$-subgroup of $G$, that is $H_1\subseteq G_1$. Thus $G_1$ contains subgroups of order $d'$, a contradiction.
\medskip

The proof of Theorem 1.1 is now complete.$\qed$
\bigskip

The following example is founded on the above proof.

\begin{example}
For $d=60=2^2\cdot 3\cdot 5$, we have $\exp_4(3)=2\nmid 1=n_2$ and therefore a finite solvable group $G$ such that $60\mid |G|$, $\pi(G)=\pi(60)$ and $G$ has no subgroup of order $60$ is \begin{equation}
{\rm SmallGroup}(360,123)={\rm AGL}(1,9)\times C_5=\left(C_3^2\rtimes C_8\right)\times C_5.\nonumber
 \end{equation}
\end{example}

\section{Proof of Theorem 1.2}

The proof of Theorem 1.2 follows the same steps as the proof of Theorem 1.1 in \cite{9}. It is based on the next lemma which is a
consequence of Proposition outlined on p. 863 of \cite{12}.

\begin{lemma}
Let $(x_n)_{n\geq 1}$ be a sequence of positive real numbers such that $\lim_{n\rightarrow\infty}x_n=0$ and $\sum_{n=1}^{\infty}x_n$ is divergent. Then the set containing the sums
of all finite subsequences of $(x_n)_{n\geq 1}$ is dense in $[0,\infty)$.
\end{lemma}

We will also need the following lemma.

\begin{lemma}
Let $p,q$ be two primes such that $q$ is odd and $q\mid p+1$, and let
\begin{equation}
G_{p,q}^n=\left(C_p^2\rtimes C_q\right)\times C_q^n,\, n\geq 0.\nonumber
\end{equation}Then 
\begin{equation}
d_{CLT}(G_{p,q}^n)=\frac{3n+5}{3n+6}\,.\nonumber
\end{equation}
\end{lemma}

\begin{proof}
It is easy to check that $G_{p,q}^n$ has subgroups of all possible orders except $p\cdot q^{n+1}$.
\end{proof}

We are now able to prove our second main result.

\bigskip\noindent{\bf Proof of Theorem 1.2.} 
Let $I=\{n_1,...,n_k\}\subset\mathbb{N}$ and let $q_1,...,q_k$ be distinct odd primes. By Dirichlet's theorem, we can choose distinct primes $p_1,...,p_k$ such that $q_i\mid p_i+1$, for all $i=1,...,k$. We remark that these primes can be chosen such that 
\begin{equation}
\{p_i,q_i\}\cap\{p_j,q_j\}=\emptyset, \mbox{ for all } i\neq j.\nonumber
\end{equation}Since $d_{CLT}$ is multiplicative, Lemma 3.2 shows that 
\begin{equation}
d_{CLT}\left(\prod_{i=1}^kG_{p_i,q_i}^{n_i}\right)=\prod_{i=1}^k\frac{3n_i+5}{3n_i+6}\nonumber
\end{equation}and so 
\begin{equation}
A=\left\{\prod_{n\in I}\frac{3n+5}{3n+6} \,\bigg|\, I\subset\mathbb{N}, |I|<\infty\right\}\subseteq {\rm Im}(d_{CLT}).\nonumber
\end{equation}Thus it suffices to prove that $A$ is dense in $[0,1]$. 

Consider the sequence $(x_n)_{n\geq 1}\subset (0,\infty)$, where $x_n=\ln(\frac{3n+6}{3n+5})$. Clearly, $\lim_{n\rightarrow\infty}x_n=0$. We have
\begin{equation}
\lim_{n\rightarrow\infty}\frac{x_n}{\frac{1}{n}}=\frac{1}{3}\,.\nonumber
\end{equation}Therefore, since the series $\sum_{n\geq 1}\frac{1}{n}$ is divergent, we deduce that the series
$\sum_{n\geq 1}x_n$ is also divergent. So, all hypotheses of Lemma 3.1 are satisfied, implying that
\begin{equation}
\overline{\left\{\sum_{n\in I}x_n \,\bigg|\, I\subset\mathbb{N}^*, |I|<\infty\right\}}=[0,\infty).\nonumber
\end{equation}This means
\begin{equation}
\overline{\left\{\ln\left(\prod_{n\in I}\frac{3n+6}{3n+5}\right) \,\bigg|\, I\subset\mathbb{N}^*, |I|<\infty\right\}}=[0,\infty)\nonumber
\end{equation}or equivalently
\begin{equation}
\overline{\left\{\prod_{n\in I}\frac{3n+6}{3n+5} \,\bigg|\, I\subset\mathbb{N}^*, |I|<\infty\right\}}=[1,\infty).\nonumber
\end{equation}Then
\begin{equation}
\overline{\left\{\prod_{n\in I}\frac{3n+5}{3n+6} \,\bigg|\, I\subset\mathbb{N}^*, |I|<\infty\right\}}=[0,1]\nonumber
\end{equation}and consequently
\begin{equation}
\overline{A}=[0,1].\nonumber
\end{equation}

The proof of Theorem 1.2 is now complete.$\qed$
\bigskip

Finally, we note that for fixed $p,q$, the above groups $(G_{p,q}^n)_{n\geq 0}$ are non-CLT and satisfy $\lim_{n\rightarrow\infty}d_{CLT}(G_{p,q}^n)=1$, providing a direct proof of Corollary 1.4.
\bigskip

\noindent{\bf Acknowledgements.} The author is grateful to the reviewer for remarks which improve the previous version of the paper.

\vspace*{3ex}\small

\hfill
\begin{minipage}[t]{5cm}
Marius T\u arn\u auceanu \\
Faculty of  Mathematics \\
``Al.I. Cuza'' University \\
Ia\c si, Romania \\
e-mail: {\tt tarnauc@uaic.ro}
\end{minipage}

\end{document}